\documentclass[11 pt]{article}
\usepackage{amsmath, enumerate, amsthm, amssymb,verbatim,fullpage,multicol,cases}
\usepackage{graphicx}
\usepackage{hyperref}
\usepackage[table]{xcolor}
\usepackage[T1]{fontenc}
\usepackage{appendix}
\usepackage[utf8]{inputenc}
\usepackage{soul}
\usepackage{braket}
\usepackage{authblk}
\usepackage{soul}
\usepackage{enumitem}

\usepackage{multirow}
\usepackage{float}
\usepackage{verbatim}
\restylefloat{table}
\newtheorem{theorem}{Theorem}[section]

\newtheorem{corollary}[theorem]{Corollary}

\theoremstyle{definition}
  \newtheorem{definition}[theorem]{Definition}
  \newtheorem{example}[theorem]{Example}
\usepackage[mathscr]{euscript}
\usepackage{float}

\definecolor{lightgray}{gray}{0.85}
\definecolor{LightCyan}{rgb}{0.88,1,1}
\definecolor{midgray}{gray}{.6}

\theoremstyle{definition}

\newtheoremstyle{named}{}{}{\itshape}{}{\bfseries}{.}{.5em}{\thmnote{#3 }#1}
\theoremstyle{named}

\def\S{\mathfrak{S}}

\def\P{\mathcal{P}}
\def\pS{p_S}

\newcommand{\tcb}[1]{\textcolor{blue}{#1}}
\newcommand{\tcr}[1]{\textcolor{red}{#1}}

\title{A proof of the peak polynomial positivity conjecture}
 \author[1]{Alexander Diaz-Lopez\thanks{\textcolor{blue}{\href{mailto:adiazlo1@swarthmore.edu}{adiazlo1@swarthmore.edu}}}}
 \author[2]{Pamela E. Harris\thanks{\textcolor{blue}{\href{mailto:pamela.e.harris@williams.edu}{pamela.e.harris@williams.edu}}.}}
 \author[3]{Erik Insko\thanks{\textcolor{blue}{\href{mailto:einsko@fgcu.edu}{einsko@fgcu.edu}}}}
 \author[4]{Mohamed Omar\thanks{\textcolor{blue}{\href{mailto:omar@g.hmc.edu}{omar@g.hmc.edu}}}}
 \affil[1]{Department of Mathematics and Statistics, Swarthmore College}
 \affil[2]{Department of Mathematics and Statistics, Williams College}
 \affil[3]{Department of Mathematics, Florida Gulf Coast University}
 \affil[4]{Department of Mathematics, Harvey Mudd College}
 
\date{ }
\begin{document}
  \maketitle

 \begin{abstract}
We say that a permutation $\pi=\pi_1\pi_2\cdots \pi_n \in 
\mathfrak{S}_n$ has a peak at index $i$ if $\pi_{i-1} < \pi_i > \pi_{i+1}$.  Let $\P(\pi)$
denote the set of indices where $\pi$ has a peak.  Given a set $S$ of positive
integers, we define $\P_S(n)=\{\pi\in\mathfrak{S}_n:\P(\pi)=S\}$. In 2013
Billey, Burdzy, and Sagan showed that for subsets of positive
integers $S$ and sufficiently large $n$,
$| \P_S(n)|=p_S(n)2^{n-|S|-1}$ where $p_S(x)$ is a polynomial depending
on $S$. They gave a 
recursive formula for $p_S(x)$ involving an alternating sum, and they conjectured that the coefficients 
of $p_S(x)$ expanded
in a binomial coefficient basis centered at $\max(S)$ are all
nonnegative. In this paper we introduce a new recursive formula for $|\P_S(n)|$ without alternating sums, and we use this 
recursion to prove that their conjecture is true.  
\end{abstract}

%%%% Section 1 -- Introduction %%%%%%%%%%%%%%%
\section{Introduction}
Let $[n]:=\{1,2,\ldots, n\}$  and let $\S_n$ denote the symmetric group on $n$ letters.
Let $\pi=\pi_1\pi_2\ldots \pi_n$ denote the one-line notation for $\pi\in\S_n$. We say that $\pi$ has a \emph{peak} at index $i$ if 
$\pi_{i-1}<\pi_i>\pi_{i+1}$
and define the peak set of a permutation $\pi$ to be the set: \[ \P(\pi) = \{i \in [n]\, \vert \, \mbox{ $\pi$ has a peak at $i$}\}. \]
Given a subset $S \subseteq [n]$
 we denote
 the set of all permutations with peak set $S$ by
\[\P_S(n) = \{ \pi \in \S_n \, \vert \,  \P(\pi) = S\}.\]   
Whenever $\P_S(n)\neq \emptyset$, we say $S \subseteq [n]$ is \emph{$n$-admissible} or simply \emph{admissible} when the $n$ is understood. If $S$ is $n$-admissible then it is $k$-admissible for any $k \geq n$.

Billey, Burdzy and Sagan first studied the subsets $\P_S(n) \subset \mathfrak{S}_n$ for $n$-admissible sets $S$ in 2013 
\cite{BBS13}.  Their work was motivated by a problem in probability theory which explored the relationship between mass 
distribution on graphs and random permutations with specific peak sets \cite{BBPS15}.  One of their foundational results established that for an $n$-admissible set $S$ 
	\begin{equation}
		 | \P_S(n) |\ = \ p_S(n) 2^{n-| S|-1} \label{eq:BBZ}
	\end{equation}
where $p_S(x)$ is a polynomial depending on $S$, which they called the \emph{peak polynomial} of $S$. 
It was shown that $p_S(x)$ has degree $\max(S)-1$, and that $p_S(x)$ takes on integral values when evaluated at integers 
\cite[Theorem 1.1]{BBS13}.  
Similar observations were made for peak polynomials in  other classical Coxeter groups (see the work of Castro-Velez, 
Diaz-Lopez, Orellana, Pastrana \cite{CV14} and Diaz-Lopez, Harris, Insko, and Perez-Lavin \cite{DLHIP}).

Using the method of finite differences, Billey, Burdzy, and Sagan gave closed formulas for the peak polynomials $p_S(x)$ 
in various special cases.
The \emph{finite forward difference} operator $\Delta$ is a linear operator defined by $ (\Delta 
f)(x)=f(x+1)-f(x)$. Iterating this 
operator gives higher order differences defined by \[ (\Delta^jf)(x)=(\Delta^{j-1}f)(x+1)-(\Delta^{j-1}f)(x). \] Using 
Newton's forward difference formula, Billey, Burdzy, and Sagan expanded $p_S(x)$ in the binomial basis centered at $k$ as 
\begin{equation} p_S(x)=\displaystyle\sum_{j=0}^{\max(S)}(\Delta^jp_S)(k)\binom{x-k}{j} \label{eqn:binomial}\end{equation}
and conjectured that for any admissible set $S$ with $m=\max(S)$ 
each coefficient $(\Delta^jp_S)(m)$ is a positive integer for $1\leq j\leq m-1$ \cite[Conjecture 14]{BBS13}. This 
conjecture has become known as the \emph{positivity conjecture} for peak polynomials.
\begin{example} Below is a table of forward 
differences for the peak polynomial $p_{\{4,6\}}(x)$.  The $(j,k)$ entry in this table is the 
coefficient $\Delta^j(p_S(k))$ of $\binom{x-k}{j}$ in the expansion of $p_S(x)$ in the binomial basis centered at $k$. 

\begin{table}[h!]
\begin{center}
 \begin{tabular}{| l | rrrrrrr |} \hline
 $j,k$  &    0         & 1 & 2 & 3 & 4 & 5     & 6  \\ \hline   
  0    & \tcr{4} & 2 & 2 & 2 & 0 & $-3$ & \tcb{0} \\
  1     & \tcr{$-2$} & 0 & 0 & $-2$ & $-3$ & {3} & \tcb{25} \\
  2    & \tcr{2} & 0 & $-2$ & $-1$ & {6} & 22 & \tcb{50} \\
  3     & \tcr{$-2$} & $-2$ & 1 & {7} & 16 & 28 & \tcb{43} \\
  4     & \tcr{0} & 3 & {6} & 9 & 12 & 15 & \tcb{18} \\
  5     & \tcr{3} & {3} & 3 & 3 & 3 & 3 & \tcb{3} \\
  6     & \tcr{0} & 0 & 0 & 0 & 0 & 0 & \tcb{0} \\ \hline
\end{tabular}  \end{center}
\caption{Forward Difference Table for the Peak Polynomial $p_{\{4,6\}}(x)$}
\end{table}
For example, we expand $p_{\{4,6\}}(x)$ in the binomial bases centered at $0$ and $6$ as 
\begin{align*} p_{\{4,6\}}(x) & = \tcr{4} { x \choose 0 } \tcr{- 2} {x \choose 1} + \tcr{2} { x \choose 2}  \tcr{-
2} { x 
\choose 3}  + \tcr{0} {x \choose 4} + \tcr{3} {x \choose 5}+\tcr{0} {x \choose 6} \\
& \hspace{-8mm}=  \tcb{0} { x-6 
\choose 0}  + \tcb{25} { x-6 \choose 1}  + \tcb{50} {x-6 \choose 2}  + \tcb{43} { x-6 \choose 3}  + \tcb{18}{ x-6 \choose 4} + 
\tcb{3} {x-6 \choose 5}+\tcb{0} {x-6 \choose 6} . \end{align*}
\end{example}

Billey, Burdzy, and Sagan proved the positivity conjecture holds when $|S| \leq 1$, verified it computationally for all $2^m$ subsets containing a largest value $m =\max(S) =20$, and showed that $p_S(m)=0$ for 
any set $S$ \cite[Lemma~15]{BBS13}.
In 2014, Billey, Fahrbach, and Talmage posed a stronger conjecture bounding the moduli of the roots of $p_S (x)$, 
which they verified for all peak sets with $\max(S) \leq 15$ \cite[Conjecture 1.5]{BFT16}. They also discovered a computationally 
efficient recursive algorithm for computing $p_S(x)$, and  showed that $p_S (k) > 0$ for $k > m$ and that the positivity 
conjecture holds in several special cases, including when the position of the last peak of $S$ is three more than the position of 
the penultimate peak \cite[Lemmas~5.4~and~4.6]{BFT16}.

Our main result is the following theorem which proves the positivity conjecture in all cases.

%%%%%% Main Theorem %%%%%%%%%%%%
\begin{theorem} \label{thm:main}
If $S\subseteq [n]$ is a non-empty admissible set with $m=\max(S)$, then 
$(\Delta^j p_S)(k)> 0$ for all $1 \leq j\leq m-1$ and $k\geq m$, and $(\Delta^m p_S)(x)=0$.
\end{theorem}
%%%%%%%%%%%%%%%%%%%%%%%%%%%%%%

We prove Theorem \ref{thm:main} at the end of Section \ref{sec:jefe}.
As a consequence of this theorem and Equation~(\ref{eqn:binomial}), 
if $S$ is an $n$-admissible set and $k > \max(S)$, then $p_S(k) > 0$. Positivity of coefficients in a given binomial basis is a phenomenon that occurs throughout combinatorics.  A particular illuminating example comes from Ehrhart Theory.  For a $d$-dimensional integral convex polytope $P$, recall that $i_P(n)$ is the number of integer points in the $n$-th dilation of $P$.  Ehrhart proved that $i_P(n)$ is a polynomial in $n$ of degree $d$, so classical techniques in generating functions establish that 
$i_P(n) = \sum_{j=0}^d h^*_j \binom{n+d-j}{d}.$
The vector $(h^*_0,h^*_1,\ldots,h^*_d)$ is called the $h^*$-vector of $P$, 
and a celebrated theorem of Stanley confirms that $h^{*}_j \geq 0$ for all $j$, \cite[Theorem 1]{Stanley80}.  

In addition to positivity, we have verified that the coefficients $\Delta^{j}p_S(m)$ are $\log$-concave for all $1 \leq j\leq m-1$ and all admissible sets $S$ with $m=\max(S)\leq 20$, and we suspect that $\log$-concavity holds in general. We note that $\log$-concavity along with our positivity result would imply the unimodality of the coefficients $\Delta^{j}p_S(m)$ for $1\leq j\leq m-1$. If unimodality is not true in general, a related problem would be classifying peak sets for which unimodality holds.  Such problems are a major theme throughout combinatorics (for instance, they are central in Ehrhart Theory \cite{BraunD}) and could lead to many interesting and fruitful combinatorial questions.

In addition, Theorem \ref{thm:main} provides supporting evidence for Billey, Fahrbach, and Talmage's stronger conjecture bounding the moduli of the zeros of peak polynomials \cite[Conjecture 1.5]{BFT16}.  After stating that conjecture, they noted that Ehrhart, chromatic, and Hilbert polynomials are all examples of polynomials with integer coefficients (in some basis) whose roots are bounded in the complex plane \cite{BDLDPS,BraunD,Brenti92,BRW,BCKV,Pfeifle,RV}. Their conjecture suggests that peak polynomials fit into the family of polynomials sharing these properties.

%%% Section 2 %%%%%%%%%%%%%%%%%%%%%%%%%%%%%%%%%%%%%%%%%%%%%%%%%

\section{Peak polynomial positivity result}\label{sec:jefe}
We begin with a definition which is used throughout the rest of this paper.
\begin{definition}
Let $S=\{i_1,i_2,\ldots, i_s\}\subseteq [n+1]$ with $i_1<i_2<\ldots<i_s$ be an $(n+1)$-admissible set. For $1\leq \ell\leq s$ define
\begin{align*}
S_{i_\ell}&=\{i_1,i_2,\ldots,i_{\ell-1},i_\ell-1,i_{\ell+1}-1,i_{\ell+2}-1,\ldots,i_{s}-1\}, \\
\widehat{S}_{i_\ell}&=\{i_1,i_2,\ldots,i_{\ell-1},\widehat{i_\ell},i_{\ell+1}-1,i_{\ell+2}-1,\ldots,i_{s}-1\}, 
\end{align*} 
where the notation $\widehat{i_\ell}$ means that the element $i_\ell$ has been omitted from the set.
\end{definition}
In general, the sets $S_{i_\ell}$ might not be $n$-admissible as they may contain two adjacent integers when $i_{\ell-1}$ 
and $i_{\ell}-1=i_{\ell-1}+1$. 
However, the sets $\widehat{S}_{i_\ell}$ are always $n$-admissible.
\begin{example}
If $S=\{3,5,8\}\subseteq[9]$, then \\
\begin{center}
\begin{tabular}{lll}
$S_3={\{2,4,7\}}$,&$S_{5}=\{3,4,7\}$,&$S_8=\{3,5,7\}$,\\
$\widehat{S}_{3}=\{4,7\}$, &$\widehat{S}_{5}=\{3,7\}$,&$\widehat{S}_{8}=\{3,5\}$.
\end{tabular}
\end{center}
The sets $S_3, S_8, \widehat{S}_3, \widehat{S}_5, \widehat{S}_8$ are 8-admissible whereas $S_5$ is not.
\end{example}
\noindent
Our first result describes a recursive construction of the set $P_S({q}+1)$ from disjoint subsets in $\S_{{q}}$.
\begin{theorem}\label{thm:uno}Let $S=\{i_1,i_2,\ldots, i_s\}\subseteq [{n}+1]$ with $i_1<i_2<\ldots<i_s$ be an 
$({n}+1)$-admissible set. 
Then for $q\geq \max(S)$
\begin{align}
|\P_S(q+1)|&=2|\P_S(q)|+2\displaystyle\sum_{\ell=1}^{s}|\P_{S_{i_\ell}}(q)|+\displaystyle\sum_{\ell=1}^{s}|\P_{\widehat{S}_{i_\ell}}(q)|.
\end{align}
\end{theorem}
\begin{proof}
We recursively build all permutations in $\P_S(q+1) \subset \S_{q+1}$ from permutations in $\S_q$ by inserting the number $q+1$ (in different positions) in the permutations of ${\S}_q$. Let $\pi\ =\ \pi_1 \cdots \pi_q$ be a permutation in ${\S}_q$ {and consider the following} five cases:\\
	\textbf{Case 1:} If $\pi\in\P_S(q)$, then by inserting $q+1$ after $\pi_q$ we create the permutation \[\hat{\pi}=\pi_1\pi_2\cdots\pi_q(q+1)\in\P_S(q+1).\]
	\textbf{Case 2:} If $\pi\in\P_S(q)$, then by inserting $q+1$ between $\pi_{i_{s}-1}$ and $\pi_{i_{s}}$ we {create the} permutation \[\hat{\pi}= \pi_1 \cdots \pi_{i_{s}-1} (q+1) \pi_{i_s} \cdots \pi_{q}\in\P_S(q+1).\] 
	\textbf{Case 3:} If $\pi\in \P_{S_{i_\ell}}(q)$ for any $1\leq \ell \leq s$, then by inserting $q+1$ between $\pi_{i_{\ell}-1}$ and $\pi_{i_{\ell}}$ we create the permutation \[\hat{\pi}=\pi_1 \cdots \pi_{i_{\ell}-1} (q+1)\pi_{i_{\ell}}\cdots \pi_q\in\P_{S}(q+1).\]
	\textbf{Case 4.1:} If $\pi\in\P_{S_{i_\ell}}(q)$ and $1<\ell\leq s$, {then} $\pi$ has a peak at position $i_{\ell -1}$ and by inserting $q+1$ between $\pi_{i_{\ell -1}-1}$ and $\pi_{i_{\ell -1}}$ we create the permutation \[\hat{\pi}=\pi_1 \cdots  \pi_{i_{\ell -1}-1} (q+1)  \pi_{i_{\ell -1}} \cdots \pi_q \in \P_S(q+1).\]  
	\textbf{Case 4.2:} If $\pi\in\P_{S_{i_1}}(q)$ where $S_{i_{1}}= \{i_1-1,i_2-1, \ldots, i_s-1\}$, then by inserting $q+1$ to the left of $\pi_1$ we create the permutation  \[\hat{\pi}=(q+1)\pi_1 \cdots  \pi_q \in \P_S(q+1).\]
\textbf{Case 5:} If  $\pi\in\P_{\widehat{S}_{i_\ell}}(q)$ for any $1\leq \ell \leq s$, then  $\pi$ has no peak at position $i_{\ell}$. By inserting $q+1$ between $\pi_{i_{\ell}-1}$ and $\pi_{i_{\ell}}$ we create the permutation \[\hat{\pi}=\pi_1 \cdots \pi_{i_{\ell}-1}(q+1) \pi_{i_{\ell}} \cdots \pi_q\in\P_{S}(q+1).\] 
	
The permutations $\hat{\pi}$ created via Cases 1 through 5 are distinct elements of $\P_S(q+1)$. This is because given any two such permutations with $(q+1)$ in the same position, if we remove $q+1$ we get two permutations in $\S_q$ with distinct peak sets. In fact, we show that $\P_S(q+1)$ is precisely the union of the permutations $\hat{\pi}$ appearing in Cases 1 through 5.  If this is the case, the sets being disjoint gives us

\[|\P_S(q+1)|=2|\P_S(q)| + 2\displaystyle\sum_{\ell=1}^{s}|\P_{S_{i_\ell}}(q)|+\displaystyle\sum_{\ell=1}^{s}|\P_{\widehat{S}_{i_\ell}}(q)|.\] 
	
Note that any permutation $\hat{\pi}$ in $\P_S(q+1)$ has the number $(q+1)$ in one of the following positions: $1, i_1, \ldots, i_s, q+1$. If $(q+1)$ is in position $q+1$, then removing it from the permutation {$\hat{\pi}$} yields a permutation $\pi$ in Case 1. If $(q+1)$ is in the first position, then removing it from the permutation {$\hat\pi$} yields a permutation {$\pi$} in Case 4.2. If $(q+1)$ is in position $i_\ell$ for some $1\leq \ell\leq s$, then removing it {from the permutation $\hat\pi$} leads to three possibilities: a permutation with a peak at position $i_\ell$ (Cases 2 and 4.1), a permutation with a peak at position $i_{\ell}-1$ (Case 3), or a permutation without a peak at positions $i_{\ell}-1$ or $i_\ell$ (Case 5). Thus we have created all permutation in $\P_S(q+1)$ via the constructions in Cases 1-5. 
\end{proof}

The following result plays a key role in the proof of Theorem \ref{thm:main}.
\begin{corollary}\label{cor:uno}Let $S=\{i_1,i_2,\ldots, i_s\}\subseteq [n+1]$ with $i_1<i_2<\ldots<i_s$ be an $(n+1)$-admissible set. 
Then the following equality of polynomial holds
\begin{align}
\Delta\pS(x)&\ =\ \displaystyle\sum_{\ell=1}^{s}p_{S_{i_\ell}}(x)+\displaystyle\sum_{\ell=1}^{s}p_{\widehat{S}_{i_\ell}}(x).
\end{align}
\end{corollary}

\begin{proof}
Let $m=\max(S)$.  It suffices to show that the two polynomials agree at infinitely many values, and to do so we show that for any $q \geq m$, 
\begin{align}
\Delta\pS(q)&\ =\ \displaystyle\sum_{\ell=1}^{s}p_{S_{i_\ell}}(q)+\displaystyle\sum_{\ell=1}^{s}p_{\widehat{S}_{i_\ell}}(q).
\end{align}
Observe that for such $q$, substituting Equation \eqref{eq:BBZ} appropriately into Theorem \ref{thm:uno} yields
\begin{align}
2^{q-|S|}\pS(q+1)-2^{q-|S|}\pS(q)&=\displaystyle\sum_{\ell=1}^{s}2^{q-|S_{i_\ell}|}p_{S_{i_\ell}}(q)+\displaystyle\sum_{\ell=1}^{s}2^{q-|\widehat{S}_{i_\ell}|-1}p_{\widehat{S}_{i_\ell}}(q)\nonumber\\
&=2^{q-|S|}\displaystyle\sum_{\ell=1}^{s}p_{S_{i_\ell}}(q)+2^{q-|S|}\displaystyle\sum_{\ell=1}^{s}p_{\widehat{S}_{i_\ell}}(q)\label{eq1}
\end{align}
where the last equality holds since $|S_{i_\ell}|=|S|$ and $|\widehat{S}_{i_\ell}|=|S|-1$ for all $1\leq \ell\leq s$. 
The result follows from multiplying Equation \eqref{eq1} by ${1}/{2^{q-|S|}}$.
\end{proof}

%%%%%%%%%  MAIN RESULT %%%%%%%%%%%%%%%%%

We are now ready to prove the positivity conjecture for peak polynomials.
\begin{proof}[Proof of Theorem \ref{thm:main}]
We induct on  $m=\max(S)$. The base case is when $S=\{2\}$. 
It is known that $p_{\{2\}}(x) =x-2$ \cite[Theorem 6]{BBS13}.  Hence, we see
$(\Delta p_{\{2\}})(x)=1 >0$, and $(\Delta^2 p_{\{2\}})(x)=0$. 
Now suppose $S$ is an arbitrary admissible set satisfying the conditions of the theorem and further suppose the theorem holds for all peak polynomials $p_T(x)$ with admissible set $T$ and $\max(T)< m$.  Let $S=\{i_1,i_2,\ldots, i_s\}\subseteq [n]$ with $i_1<i_2<\ldots<i_s$, and for $1\leq \ell\leq s$ construct the sets $
S_{i_\ell}$ and $\widehat{S}_{i_\ell}$.
From Corollary~\ref{cor:uno}, we have
\begin{align*}
\Delta\pS(x)&\ =\ \displaystyle\sum_{\ell=1}^{s}p_{S_{i_\ell}}(x)+ \displaystyle\sum_{\ell=1}^{s}p_{\widehat{S}_{i_\ell}}(x).
\end{align*}
For any $1\leq j \leq m-1$, 
\begin{align}
\Delta^j\pS(x) &\ =\ \Delta^{j-1} \left( \Delta\pS(x) \right) \ =\ \displaystyle\sum_{\ell=1}^{s}\Delta^{j-1} 
p_{S_{i_\ell}}(x)+ \displaystyle\sum_{\ell=1}^{s}\Delta^{j-1} p_{\widehat{S}_{i_\ell}}(x).\label{eq:deltaj}
\end{align}
Let $k \geq m$.  Recall that for all $\ell \in \{1,2,\ldots,s\}$ we have $k\geq m > \max(S_{i_\ell})$ and $k \geq m > \max(\widehat{S}_{i_\ell})$.  
Consequently, since the degree of $p_T(x)$ is $\max(T)-1$ for any admissible peak set $T \subseteq [n]$, we have that 
	$\deg(p_{\widehat{S}_{i_\ell}}(x))=m-2=\deg( p_{S}(x)) -1$ for $1\leq \ell <s$,  $\deg(p_{S}(x)) > \deg( p_{\widehat{S}_{i_s}}(x))$, and $\deg(p_{S}(x)) > \deg( p_{S_{i_\ell}}(x))$ for all $1 \leq \ell \leq s$.  
By induction it follows that for $k \geq m$ 
	\[\Delta^{j-1} p_{S_{i_\ell}}(k) \geq 0 \text{ for } 1 \leq \ell \leq  s, \quad \Delta^{j-1} p_{\widehat{S}_{i_s}}(k) \geq 0, \;\; \text{ and }\;\; \Delta^{j-1} p_{\widehat{S}_{i_\ell}}(k)>0 \text{ for } 1 \leq \ell < s.\]
	 From Equation \eqref{eq:deltaj} we see that $\Delta^j\pS(k) > 0$. 
Finally, we claim that $(\Delta^mp_S)(x)=0$. Since $\deg(p_S(x))=m-1$ and the 
operator $\Delta$ decreases the degree by one, we see that $(\Delta^{m-1}p_S)(x)=c$ is a positive constant and 
$(\Delta^{m}p_S)(x)=0.$ 
\end{proof}

\section{Acknowledgments}
We thank Sara Billey for introducing us to this problem, for many helpful conversations, and for insightful comments regarding an earlier version of this manuscript. We 
thank Lucas Everham, Vince Marcantonio, Karen Marino, Darleen Perez-Lavin, and Nora Youngs for beneficial conversations regarding this topic. 
This work was 
supported through mini-collaboration travel grants via awards from the National Science Foundation (DMS-1545136) and the National 
Security Agency (H98230-15-1-0091). The second author gratefully acknowledges travel support from the United States Military 
Academy. The third author gratefully acknowledges funding support from the Seidler Student/Faculty Undergraduate 
Scholarly Collaboration Fellowship Program at Florida Gulf Coast University.

\bibliography{DEHIM} 

\bibliographystyle{plain}

\end{document}